\documentclass{article}
\usepackage{cite}
\usepackage{amsmath,amssymb,amsthm}
\usepackage{titlesec,hyperref}
\usepackage{color}
\usepackage{fancyhdr}
\usepackage[margin=2.5cm]{geometry}
\usepackage{enumerate}

\newtheorem{theo}{Theorem}[section]
\newtheorem{lemm}[theo]{Lemma}

\newtheorem{prop}[theo]{Proposition}

\numberwithin{equation}{section}

\allowdisplaybreaks

\begin{document}
\title{A new result for the local well-posedness of the generalized Camassa-Holm equations in critial Besov spaces $B^{\frac{1}{p}}_{p,1},1\leq p<+\infty$}

\author{
    Xi $\mbox{Tu}^1$ \footnote{email: 904817751@qq.com} \quad and \quad
    Zhaoyang $\mbox{Yin}^{1,2}$ \footnote{email: mcsyzy@mail.sysu.edu.cn}
    \quad and\quad
    Yingying $\mbox{Guo}^{3}$ \footnote{email: guoyy35@fosu.edu.cn}\\
    $^1\mbox{School}$ of Mathematics and Big Data, Foshan University,\\
    Foshan, 528000, China\\
    $^2\mbox{Faculty}$ of Information Technology,\\
    Macau University of Science and Technology, Macau, China\\
    $^3\mbox{School}$ of Mathematics and Big Data, Foshan University,\\
    Foshan, 528000, China
}

\date{}
\maketitle
\hrule

\begin{abstract}
This paper is devoted to studying the local well-posedness (existence,uniqueness and continuous dependence) for the generalized  Camassa-Holm equations in critial Besov spaces $B^{\frac{1}{p}}_{p,1}$ with $1\leq p<+\infty$, which improves the previous index $s> \max\{\frac{1}{2},\frac{1}{p}\}$ or $s=\frac{1}{p},\ p\in[1,2],\ r=1$ in \cite{linb,tu-yin4}. The main difficulty is to prove the uniqueness, which need to use the Moser-type inequality. To overcome the difficulty, we use the Lagrange coordinate transformation to obtain the uniqueness.
\end{abstract}

\vspace*{10pt}

\tableofcontents

\section{Introduction}
In this paper we consider the Cauchy problem for the following generalized Camassa-Holm equation,
\begin{align}\label{E1}
\left\{
\begin{array}{ll}
u_t-u_{txx}=\frac{1}{2}(3u_{x}^2-2u_{x}u_{xxx}-u_{xx}^2),~~~~  t>0,\\[1ex]
u(0,x)=u_{0}(x),
\end{array}
\right.
\end{align}
which can be rewritten as
\begin{align}\label{E2}
\left\{
\begin{array}{ll}
m=u-u_{xx},\\[1ex]
m_t-u_xm_x=-\frac{1}{2}m^2+um+\frac{1}{2}u_{x}^2-\frac{1}{2}u^2
, ~~ t>0,\\[1ex]
m(0,x)=u(0,x)-u_{xx}(0,x)=m_0(x).
\end{array}
\right.
\end{align}
The equation (\ref{E1}) was proposed recently by Novikov in \cite{n1}. He showed that the equation (1.1) is integrable by using as definition of integrability the existence of an infinite hierarchy of quasi-local higher symmetries \cite{n1} and it belongs to the following class \cite{n1}:
\begin{align}\label{E02}
(1-\partial^2_x)u_t=F(u,u_x,u_{xx},u_{xxx}),
\end{align}
which has attracted much interest, particularly in the possible integrable members of (\ref{E02}).

The most celebrated integrable members of (\ref{E02}) which have quadratic nonlinearity are  the well-known Camassa-Holm (CH) equation \cite{Camassa} and the famous Degasperis-Procesi (DP) equation \cite{D-P}:
\begin{align}
(1-\partial^2_x)u_t=3uu_x-2u_{x}u_{xx}-uu_{xxx},\\
(1-\partial^2_x)u_t=4uu_x-3u_{x}u_{xx}-uu_{xxx}.
\end{align}
Both the CH equation and the DP
equation can be regarded as a shallow water wave equation \cite{Camassa, Constantin.Lannes,D-G-H}.  They are completely integrable with a bi-Hamiltonian structure \cite{Constantin-E,D-H-H,Fokas}. That means that the system can be transformed into a linear flow at constant speed in suitable action-angle variables (in the sense of infinite-dimensional Hamiltonian systems), for a large class of initial data \cite{Camassa,Constantin-P,Constantin.mckean,D-H-H}.
 It admits exact the single peakon solutions and and the multi-peakon solutions, which are orbitally stable \cite{Constantin.Strauss}.
 It is worth mentioning that the peaked solitons present the characteristic for the traveling water waves of greatest height and largest amplitude and arise as solutions to the free-boundary problem for incompressible Euler equations over a flat bed, cf. \cite{Camassa.Hyman,Constantin2,Constantin.Escher4,Constantin.Escher5,Toland}.
Another remarkable feature of the CH equation and the DP equation is the so-called wave breaking phenomena  \cite{Constantin,Constantin.Escher3,liy}
The main difference between DP equation and CH equation is that DP equation has short waves \cite{Lu} and the periodic shock waves \cite{E-L-Y}.

Concerning the local well-posedness and ill-posedness for the Cauchy problem of the CH equation in Sobolev spaces and Besov spaces, we refer to \cite{Constantin.Escher,Constantin.Escher2,d1,G-L-M-Y,L-Y,Guillermo}. Global strong solutions to the CH equation were discussed in \cite{Constantin,Constantin.Escher,Constantin.Escher2}. And the finite time blow-up strong solutions to the CH equation were proved in \cite{Constantin,Constantin.Escher,Constantin.Escher2,Constantin.Escher3}. It was shown that there exist the global weak solutions to the CH equation \cite{Constantin.Molinet, Xin.Z.P} and the global conservative and dissipative solutions of CH equation \cite{Bressan.Constantin,Bressan.Constantin2}.

 The
local well-posedness of the Cauchy problem of the DP equation in Sobolev spaces and Besov spaces  were investigated in
\cite{G-L,H-H,y1}. Similar to the CH equation, It was shown that there exist the global strong solutions
\cite{L-Y1,y2,y4} and finite time blow-up solutions
\cite{E-L-Y1, E-L-Y,L-Y1,L-Y2,y1,y2,y3,y4} to the
DP equation. The global weak
solutions was established in \cite{C-K,E-L-Y1,y3,y4}.
\par


The third celebrated integrable member of (\ref{E02}) which has cubic nonlinearity is the known Novikov equation \cite{n1}:
\begin{align}
(1-\partial^2_x)u_t=3uu_{x}u_{xx}+u^2u_{xxx}-4u^2u_x.
\end{align}
It was showed that the Novikov equation is integrable, possesses a bi-Hamiltonian structure, and admits exact peakon solutions $u(t,x)=\pm\sqrt{c}e^{|x-ct|}$ with $c>0$ \cite{Hone}.\\

 The local well-posedness for the Novikov equation in Sobolev spaces and Besov spaces was investigated in \cite{Wu.Yin2,Wu.Yin3,Wei.Yan,Wei.Yan2}. Wu and Yin proved the global existence of strong solutions under some sign conditions \cite{Wu.Yin2}.  Yan, Li and  Zhang studied
 the blow-up phenomena of the strong solutions  \cite{Wei.Yan2}. The global weak solutions for the Novikov equation was established in \cite{Laishaoyong,Wu.Yin}.

Recently, the Cauchy problem of (\ref{E1}) in the Besov spaces $B^{s}_{p,r},~s>max\{2+\frac{1}{p},\frac{5}{2}\}$ and the critical Besov space $B^{\frac{1}{2}}_{2,1}$ has been studied in \cite{linb,tu-yin4}. The global weak solution of (\ref{E1}) was established in \cite{tu-yin6}.
To our best knowledge,  there is no paper concerning the Cauchy problem of (\ref{E1}) in the critical Besov space $B^{\frac{1}{p}}_{p,1},1\leq p<+\infty$, which is we shall investigate in this paper.

The main difficulty is to prove the uniqueness. For instance, one should use the following Moser-type inequality
\begin{align}\label{mosher}
\|fg\|_{{B}^{s_1+s_2-\frac{d}{p}}_{p,1}}\leq C\|f\|_{{B}^{s_1}_{p,1}}\|g\|_{{B}^{s_2}_{p,1}},\quad s_1,s_2\leq \frac{d}{p}, s_1+s_2>d\max\{0,\frac{2}{p}-1\}
\end{align}
to estimate (\ref{E1}).
That is why one need the condition $s>\max\{\frac{5}{2},2+\frac{1}{p}\}$. To overcome the difficulty, we use the Lagrange coordinate transformation in this paper to investigate the uniqueness for the generalized Camassa-Holm equation. Indeed, combining with the estimation of the characteristic $y(t,\xi)$, we will obtain the uniqueness without using \eqref{mosher}.

The rest of our paper is as follows. In the second section, we introduce some preliminaries which will be used in the sequel. In the third section, we give the proof of Theorem \ref{Thm1} by using the Lagrangian coordinate transformation.

\vspace*{2em}
\section{Preliminaries}
\par
In this section, we first recall some basic properties on the Littlewood-Paley theory, which can be found in \cite{B.C.D}.

Let $\chi$ and $\varphi$ be a radical, smooth, and valued in the interval $[0,1]$, belonging respectively to $\mathcal{D}(\mathcal{B})$ and $\mathcal{D}(\mathcal{C})$, where $\mathcal{B}=\{\xi\in\mathbb{R}^d:|\xi|\leq\frac 4 3\},\ \mathcal{C}=\{\xi\in\mathbb{R}^d:\frac 3 4\leq|\xi|\leq\frac 8 3\}$.
Denote $\mathcal{F}$ by the Fourier transform and $\mathcal{F}^{-1}$ by its inverse.
For any $u\in\mathcal{S}'(\mathbb{R}^d)$,  all $j\in\mathbb{Z}$, define
$\Delta_j u=0$ for $j\leq -2$; $\Delta_{-1} u=\mathcal{F}^{-1}(\chi\mathcal{F}u)$; $\Delta_j u=\mathcal{F}^{-1}(\varphi(2^{-j}\cdot)\mathcal{F}u)$ for $j\geq 0$; and $S_j u=\sum_{j'<j}\Delta_{j'}u$.

Let $s\in\mathbb{R},\ 1\leq p,r\leq\infty.$ The nonhomogeneous Besov space $B^s_{p,r}(\mathbb{R}^d)$ is defined by
$$  B^s_{p,r}=B^s_{p,r}(\mathbb{R}^d)=\Big\{u\in S'(\mathbb{R}^d):\|u\|_{B^s_{p,r}}=\big\|(2^{js}\|\Delta_j u\|_{L^p})_j \big\|_{l^r(\mathbb{Z})}<\infty\Big\}.$$

The nonhomogeneous Sobolev space is defined by
$$
H^{s}=H^{s}(\mathbb{R}^d)=\Big\{u\in S'(\mathbb{R}^d):\ u\in L^2_{loc}(\mathbb{R}^d),\ \|u\|^2_{H^s}=\int_{\mathbb{R}^d}(1+|\xi|^2)^s|
\mathcal{F}u(\xi)|^2{d}\xi<\infty\Big\}. $$

The nonhomogeneous Bony's decomposition is defined by
$uv=T_{u}v+T_{v}u+R(u,v)$ with
$$T_{u}v=\sum_{j}S_{j-1}u\Delta_{j}v,\ \ R(u,v)=\sum_{j}\sum_{|j'-j|\leq 1}\Delta_{j}u\Delta_{j'}v.$$

Naturally, we introduce some properties about Besov spaces. For more details, see \cite{B.C.D}.
\begin{prop}\label{Besov}\cite{B.C.D,liy}
	Let $s\in\mathbb{R},\ 1\leq p,p_1,p_2,r,r_1,r_2\leq\infty.$  \\
	{\rm(1)} $B^s_{p,r}$ is a Banach space, and is continuously embedded in $\mathcal{S}'$. \\
	{\rm(2)} If $r<\infty$, then $\lim\limits_{j\rightarrow\infty}\|S_j u-u\|_{B^s_{p,r}}=0$. If $p,r<\infty$, then $C_0^{\infty}$ is dense in $B^s_{p,r}$. \\
	{\rm(3)} If $p_1\leq p_2$ and $r_1\leq r_2$, then $ B^s_{p_1,r_1}\hookrightarrow B^{s-d(\frac 1 {p_1}-\frac 1 {p_2})}_{p_2,r_2}. $
	If $s_1<s_2$, then the embedding $B^{s_2}_{p,r_2}\hookrightarrow B^{s_1}_{p,r_1}$ is locally compact. \\
	{\rm(4)} $B^s_{p,r}\hookrightarrow L^{\infty} \Leftrightarrow s>\frac d p\ \text{or}\ s=\frac d p,\ r=1$. \\
	{\rm(5)} Fatou property: if $(u_n)_{n\in\mathbb{N}}$ is a bounded sequence in $B^s_{p,r}$, then an element $u\in B^s_{p,r}$ and a subsequence $(u_{n_k})_{k\in\mathbb{N}}$ exist such that
	$$ \lim_{k\rightarrow\infty}u_{n_k}=u\ \text{in}\ \mathcal{S}'\quad \text{and}\quad \|u\|_{B^s_{p,r}}\leq C\liminf_{k\rightarrow\infty}\|u_{n_k}\|_{B^s_{p,r}}. $$
	{\rm(6)} Let $m\in\mathbb{R}$ and $f$ be a $S^m$-mutiplier (i.e. f is a smooth function and satisfies that $\forall\ \alpha\in\mathbb{N}^d$,
	$\exists\ C=C(\alpha)$ such that $|\partial^{\alpha}f(\xi)|\leq C(1+|\xi|)^{m-|\alpha|},\ \forall\ \xi\in\mathbb{R}^d)$.
	Then the operator $f(D)=\mathcal{F}^{-1}(f\mathcal{F})$ is continuous from $B^s_{p,r}$ to $B^{s-m}_{p,r}$.
\end{prop}
\begin{prop}\cite{B.C.D}
	Let $s\in\mathbb{R},\ 1\leq p,r\leq\infty.$
	\begin{equation*}\left\{
		\begin{array}{l}
			B^s_{p,r}\times B^{-s}_{p',r'}\longrightarrow\mathbb{R},  \\
			(u,\phi)\longmapsto \sum\limits_{|j-j'|\leq 1}\langle \Delta_j u,\Delta_{j'}\phi\rangle,
		\end{array}\right.
	\end{equation*}
	defines a continuous bilinear functional on $B^s_{p,r}\times B^{-s}_{p',r'}$. Denote by $Q^{-s}_{p',r'}$ the set of functions $\phi$ in $\mathcal{S}'$ such that
	$\|\phi\|_{B^{-s}_{p',r'}}\leq 1$. If $u$ is in $\mathcal{S}'$, then we have
	$$\|u\|_{B^s_{p,r}}\leq C\sup_{\phi\in Q^{-s}_{p',r'}}\langle u,\phi\rangle.$$
\end{prop}

The useful interpolation inequalities are given as follows.
\begin{prop}\label{prop}\cite{B.C.D,liy}
{\rm(1)} If $s_1<s_2$, $\lambda\in (0,1)$ and $(p,r)\in[1,\infty]^2$, then we have
\begin{align*}
&\|u\|_{B^{\lambda s_1+(1-\lambda)s_2}_{p,r}}\leq \|u\|_{B^{s_1}_{p,r}}^{\lambda}\|u\|_{B^{s_2}_{p,r}}^{1-\lambda},\\
&\|u\|_{B^{\lambda s_1+(1-\lambda)s_2}_{p,1}}\leq\frac{C}{s_2-s_1}\Big(\frac{1}{\lambda}+\frac{1}{1-\lambda}\Big) \|u\|_{B^{s_1}_{p,\infty}}^{\lambda}\|u\|_{B^{s_2}_{p,\infty}}^{1-\lambda}.
\end{align*}
{\rm(2)} If $s\in\mathbb{R},\ 1\leq p\leq\infty,\ \varepsilon>0$, a constant $C=C(\varepsilon)$ exists such that
$$ \|u\|_{B^s_{p,1}}\leq C\|u\|_{B^s_{p,\infty}}\ln\Big(e+\frac {\|u\|_{B^{s+\varepsilon}_{p,\infty}}}{\|u\|_{B^s_{p,\infty}}}\Big). $$
\end{prop}

We now give the 1-D Moser-type estimates which we will use in the following.
\begin{lemm}\label{product}\cite{B.C.D,liy}
The following estimates hold:\\
	{\rm(1)} For any $s>0$ and any $p,\ r$ in $[1,\infty]$, the space $L^{\infty} \cap B^s_{p,r}$ is an algebra, and a constant $C=C(s)$ exists such that
\begin{align*}
		&\|uv\|_{B^s_{p,r}}\leq C(\|u\|_{L^{\infty}}\|v\|_{B^s_{p,r}}+\|u\|_{B^s_{p,r}}\|v\|_{L^{\infty}}), \\
		&\|u\partial_{x}v\|_{B^s_{p,r}}\leq C(\|u\|_{B^{s+1}_{p,r}}\|v\|_{L^{\infty}}+\|u\|_{L^{\infty}}\|\partial_{x}v\|_{B^s_{p,r}}).
\end{align*}
	{\rm(2)} If $1\leq p,r\leq \infty,\ s_1\leq s_2,\ s_2>\frac{1}{p} (s_2 \geq \frac{1}{p}\ \text{if}\ r=1)$ and $s_1+s_2>\max(0, \frac{2}{p}-1)$, there exists $C=C(s_1,s_2,p,r)$ such that
	$$ \|uv\|_{B^{s_1}_{p,r}}\leq C\|u\|_{B^{s_1}_{p,r}}\|v\|_{B^{s_2}_{p,r}}. $$
\end{lemm}

Here is the Gronwall lemma.
\begin{lemm}\label{gwl}\cite{B.C.D}
	Let $m(t),\ a(t)\in C^1([0,T]),\ m(t),\ a(t)>0$. Let $b(t)$ is a continuous function on $[0,T]$. Suppose that, for all $t\in [0,T]$,
	$$\frac{1}{2}\frac{{d}}{{d}t}m^2(t)\leq b(t)m^2(t)+a(t)m(t).$$
	Then for any time $t$ in $[0,T]$, we have
	$$m(t)\leq m(0)\exp\int_0^t b(\tau){d}\tau+\int_0^t a(\tau)\exp\big(\int_{\tau}^tb(t'){d}t'\big){d}\tau.$$
\end{lemm}

In the paper, we also need some estimates for the following 1-D transport equation:
\begin{equation}\label{transport}
	\left\{\begin{array}{l}
		f_t+v\partial_{x}f=g,\ x\in\mathbb{R},\ t>0, \\
		f(0,x)=f_0(x).
	\end{array}\right.
\end{equation}

\begin{lemm}\label{existence}\cite{B.C.D}
	Let $1\leq p\leq\infty,\ 1\leq r\leq\infty,\ \theta> -\min(\frac 1 {p}, \frac 1 {p'})$. Let $f_0\in B^{\theta}_{p,r}$, $g\in L^1([0,T];B^{\theta}_{p,r})$, and $v\in L^\rho([0,T];B^{-M}_{\infty,\infty})$ for some $\rho>1$ and $M>0$ such that
	$$
	\begin{array}{ll}
		\partial_{x}v\in L^1([0,T];B^{\frac 1 {p}}_{p,\infty}\cap L^{\infty}), &\ \text{if}\ \theta<1+\frac 1 {p}, \\
		\partial_{x}v\in L^1([0,T];B^{\theta-1}_{p,r}), &\ \text{if}\ \theta>1+\frac 1 {p}\ or\ (\theta=1+\frac 1 {p}\ and\ r=1).
	\end{array}
	$$
	Then the problem \eqref{transport} has a unique solution $f$ in \\
	-the space $C([0,T];B^{\theta}_{p,r})$, if $r<\infty$, \\
	-the space $\Big(\bigcap_{{\theta}'<\theta}C([0,T];B^{{\theta}'}_{p,\infty})\Big)\bigcap C_w([0,T];B^{\theta}_{p,\infty})$, if $r=\infty$.
\end{lemm}

\begin{lemm}\label{priori estimate}\cite{B.C.D,liy}
	Let $1\leq p,r\leq\infty,\ \theta>-min(\frac{1}{p},\frac{1}{p'}).$
	There exists a constant $C$ such that for all solutions $f\in L^{\infty}([0,T];B^{\theta}_{p,r})$ of \eqref{transport} with initial data $f_0$ in $B^{\theta}_{p,r}$ and $g$ in $L^1([0,T];B^{\theta}_{p,r})$,
	$$ \|f(t)\|_{B^{\theta}_{p,r}}\leq \|f_0\|_{B^{\theta}_{p,r}}+\int_0^t\|g(t')\|_{B^{\theta}_{p,r}}{d}t'+\int_0^t V'(t')\|f(t')\|_{B^{\theta}_{p,r}}{d}t' $$
	or
	$$ \|f(t)\|_{B^{\theta}_{p,r}}\leq e^{CV(t)}\Big(\|f_0\|_{B^{\theta}_{p,r}}+\int_0^t e^{-CV(t')}\|g(t')\|_{B^{\theta}_{p,r}}{d}t'\Big) $$
	with
	\begin{equation*}
		V'(t)=\left\{\begin{array}{ll}
			\|\partial_{x}v(t)\|_{B^{\frac 1 p}_{p,\infty}\cap L^{\infty}},\ &\text{if}\ \theta<1+\frac{1}{p}, \\
			\|\partial_{x}v(t)\|_{B^{\theta-1}_{p,r}},\ &\text{if}\ \theta>1+\frac{1}{p}\ \text{or}\ (\theta=1+\frac{1}{p},\ p<\infty, \ r=1),
		\end{array}\right.
	\end{equation*}
	and, if $\theta=\frac 1 p-1,\ 1\leq p\leq 2,\ r=\infty,\ V'(t)=\|\partial_{x}v(t)\|_{B^{\frac 1 p}_{p,1}}$.\\
	If $f=v$, then for all $\theta>0$, $V'(t)=\|\partial_{x}v(t)\|_{L^{\infty}}$.
\end{lemm}
\begin{lemm}\label{continuous}\cite{B.C.D,liy}
Let $1\leq p<\infty$. Define $\overline{\mathbb{N}}=\mathbb{N}\cup\{\infty\}$. Suppose $f\in L^{1}\big([0,T];B^{\frac{1}{p}}_{p,1}\big)$ and $a_0\in B^{\frac{1}{p}}_{p,1}$. For $n\in\overline{\mathbb{N}}$, denote by $a^n\in C\big([0,T];B^{\frac{1}{p}}_{p,1}\big)$ the solution of
\begin{equation}\label{an}
\left\{\begin{array}{l}
\partial_{t}a^n+A^n\partial_{x}a^n=f,\\
a^n(0,x)=a_0(x).
\end{array}\right.
\end{equation}
Assume for some $\beta\in L^{1}(0,T)$, $\sup\limits_{n\in\overline{\mathbb{N}}}\|A^n\|_{B^{1+\frac{1}{p}}_{p,1}}\leq \beta(t)$.
If $A^{n}$ converges to $A^{\infty}$ in $L^{1}\big([0,T];B^{\frac{1}{p}}_{p,1}\big)$, then the sequence $\{a^n\}_{n\in\mathbb{N}}$ converges to $a^\infty$ in $C\big([0,T];B^{\frac{1}{p}}_{p,1}\big)$.
\end{lemm}

\section{Local well-posedness}
In this section, we establish local well-posedness of (\ref{E2}) in the critical Besov space $B^{\frac{1}{p}}_{p,1},1\leq p<+\infty$.
Our main result can be ststed as follows:
\begin{theo}\label{Thm1}
Let $u_{0}\in B^{2+\frac{1}{p}}_{p,1}$ $m_0=u_{0}-u_{0xx}\in B^{\frac{1}{p}}_{p,1}$ with $p\in[1,\infty)$. Then there exists a time $T>0$ such that the the generalized CH equation with the initial data $u_{0}$ is locally well-posed in the sense of Hadamard.
\end{theo}
\begin{proof}
In order to prove Theorem \ref{Thm1}, we proceed as the following five steps.\\
 \textbf{Step 1. Existence.}

  First, we construct approximate solutions which are smooth solutions of some linear equations. Starting for $m_0(t,x)\triangleq m(0,x)=m_0$, we define by induction sequences $(m_{n})_{n\in\mathbb{N}}$  by solving the following linear transport equations:
 \begin{align}\label{E00}
\left\{
\begin{array}{ll}
 \partial_{t}m_{n+1}-\partial_{x}u_{n}\partial_{x}m_{n+1}
 &=\frac{1}{2}(\partial_{x}u_n)^2-\frac{1}{2}(u_n-m_{n})^2
 \\&=u_{n}m_{n}
 +\frac{1}{2}(\partial_{x}u_n)^2-\frac{1}{2}u_n^2-\frac{1}{2}m_{n}^2
 \\&=F(m_{n},u_{n}),\\[1ex]
 m_{n+1}(t,x)|_{t=0}=S_{n+1}m_{0}.\\[1ex]
\end{array}
\right.
\end{align}

We assume that $m_n\in L^{\infty}(0,T;B^{\frac{1}{p}}_{p,1})$.
Since $B^{\frac{1}{p}}_{p,1}$ is an algebra and $B^{\frac{1}{p}}_{p,1}\hookrightarrow L^{\infty}$, we deduce that
\begin{align}\label{12}
F(m_{n},u_{n})=&\nonumber\|\frac{1}{2}(\partial_{x}u_n)^2-\frac{1}{2}(u_n-m_{n})^2\|
_{B^{\frac{1}{p}}_{p,1}}
\\\nonumber \leq&
\frac{1}{2}\|(\partial_{x}u_n)^2\|_{B^{\frac{1}{p}}_{p,1}}
+\frac{1}{2}\|(u_n-m_{n})^2\|_{B^{\frac{1}{p}}_{p,1}}
\\\nonumber \leq&
\|\partial_{x}u_n\|_{B^{\frac{1}{p}}_{p,1}}\|\partial_{x}u_n\|_{L^{\infty}}
+\|u_n-m_{n}\|_{B^{\frac{1}{p}}_{p,1}}\|u_n-m_{n}\|_{L^{\infty}}
\\ \leq&C\|m_{n}\|^2_{B^{\frac{1}{p}}_{p,1}}.
\end{align}
 which leads to
$F(m_{n},u_{n})\in L^{\infty}(0,T;B^{\frac{1}{p}}_{p,1})$.
Hence, from Lemma \ref{existence}, the equation (\ref{E00}) has a global solution $m_{n+1}$ which belongs to $C([0,T);B^{\frac{1}{p}}_{p,1})$ for all positive $T$.


We define that
$U_{n}(t)\triangleq\int^{t}_{0}\|m_n(t')\|_{B^{\frac{1}{p}}_{p,1}}dt'$. By Lemma \ref{priori estimate}, we infer that
\begin{align}\label{11}
\|m_{n+1}\|_{B^{\frac{1}{p}}_{p,1}}
\nonumber\leq& e^{C\int_{0}^{t}\|\partial^2_{x}u_n\|_{B^{\frac{1}{p}}_{p,1}}dt'}
\bigg(\|S_{n+1}m_0\|_{B^{\frac{1}{p}}_{p,1}}
\\\nonumber&+\int^t_0 e^{-C\int_{0}^{t'}\|\partial^2_{x}u_n\|_{B^{\frac{1}{p}}_{p,1}}d\tau}
\|F(m_{n},u_{n})\|_{B^{\frac{1}{p}}_{p,1}}dt'\bigg)
\\ \leq& e^{CU_{n}(t)}\bigg(\|S_{n+1}m_0\|_{B^{\frac{1}{p}}_{p,1}}
+\int^t_0e^{-CU_{n}(t')}\|F(m_{n},u_{n})\|_{B^{\frac{1}{p}}_{p,1}}dt'\bigg).
\end{align}


Fix a $T>0$ such that $2C^2T\|m_0\|_{B^{\frac{1}{p}}_{p,1}}<1.$ Similar to the proof of Theorem 3.1 in \cite{tu-yin4}, we obtain that
\begin{align}\label{013}
\|m_{n}(t)\|_{B^{\frac{1}{p}}_{p,1}}\leq \frac{C\|m_0\|_{B^{\frac{1}{p}}_{p,1}}}{1-2C^2\|m_0\|_{B^{\frac{1}{p}}_{p,1}}t}
\leq \frac{C\|m_0\|_{B^{\frac{1}{p}}_{p,1}}}{1-2C^2\|m_0\|_{B^{\frac{1}{p}}_{p,1}}T}\triangleq \mathbf{M},~~~~\forall t\in[0,T].
\end{align}

Therefore, $(m_{n})_{n \in \mathbb{N}}$ is uniformly bounded in $L^{\infty}(0,T; B^{\frac{1}{p}}_{p,1})$.

Then, we will use the compactness method for the approximating sequence $\{m_n\}_{n\in\mathbb{N}}$ to get a solution $m$ of \eqref{E2}. Since $m_n$ is uniformly bounded in $L^\infty\Big([0,T];B^{\frac{1}{p}}_{p,1}\Big)$, we can deduce from \eqref{12} that $\partial_{t}m_n$ is unformly bounded in $L^\infty\Big([0,T];B^{\frac{1}{p}-1}_{p,1}\Big)$. Thus,
\begin{align*}
	m_n\ \text{is unformly bounded in}\ C\Big([0,T];B^{\frac{1}{p}}_{p,1}\Big)\cap C^{\frac{1}{2}}\Big([0,T];B^{\frac{1}{p}-1}_{p,1}\Big).	
\end{align*}
Let $(\phi_j)_{j\in\mathbb{N}}$ be a sequence of smooth functions with value in $[0,1]$ supported in the ball $B(0,j+1)$ and equal to 1 on $B(0,j)$. Notice that the map $z\mapsto\phi_jz$ is compact from $B^{\frac{1}{p}}_{p,1}$ to $B^{\frac{1}{p}-1}_{p,1}$ by Theorem 2.94 in \cite{B.C.D}. Taking advantage of Ascoli's theorem and Cantor's diagonal process, there exists some function $m_j$ such that for any $j\in\mathbb{N}$, $\phi_jm_n$ tends to $m_j$. From that, we can easily deduce that there exists some function $m$ such that for all $\phi\in\mathcal{D}$, $\phi m_n$ tends to $\phi m$ in $C\Big([0,T];B^{\frac{1}{p}-1}_{p,1}\Big)$. Combining the uniform boundness of $m_n$ and the Fatou property for Besov spaces, we realiy obtain that $m\in L^\infty\Big([0,T];B^{\frac{1}{p}}_{p,1}\Big)$. By virtue of the interpolation, we have $\phi m_n$ tends to $\phi m$ in $C\Big([0,T];B^{\frac{1}{p}-\varepsilon}_{p,1}\Big)$ for any $\varepsilon>0$. Next, it is a routine process to prove that $m$ satifies Eq. (\ref{E2}). Thanks to the right side of the Eq. (\ref{E2}), we get $\partial_{t}m\in C\Big([0,T];B^{\frac{1}{p}-1}_{p,1}\Big)$. In sum, we obtain $m$ satisfies  (\ref{E2}) and belongs to $C\Big([0,T];B^{\frac{1}{p}}_{p,1}\Big)$.

\textbf{Step 2. uniqueness.}

Define $M(t,\xi)=m(t,y(t,\xi))$ and $U(t,\xi)=u(t,y(t,\xi))$, thus, $U_{\xi}(t,\xi)=u_x(t,y(t,\xi))y_{\xi}(t,y(t,\xi))$.

The associated Lagrangian scale of \ref{E2} is the following initial valve problem
\begin{equation}\label{ODE}
	\left\{\begin{array}{ll}
		\frac{{d}y}{{d}t}=-u_{x}\big(t,y(t,\xi)\big),&\quad t>0,\quad \xi\in\mathbb{R},\\
		y(0,\xi)=\xi,&\quad \xi\in\mathbb{R}.
	\end{array}\right.
\end{equation}

 Owing to \ref{E2} and \eqref{ODE}, we can get
\begin{align}
y(t,\xi)&=\xi-\int_0^t\frac{U_{\xi}}{y_{\xi}} d\tau,\label{yn}\\
\frac{{d}}{{d}t}y_\xi(t,\xi)&=(M-U)y_\xi ,\label{ynxi}\\
\frac{{d}}{{d}t}M(t,\xi)&=[-\frac{1}{2}m^2+um+\frac{1}{2}u_{x}^2-\frac{1}{2}u^2]\circ y\nonumber \\
&=-\frac{1}{2}M^2+UM+\frac{1}{2}(\frac{U_{\xi}}{y_\xi})^2-\frac{1}{2}U^2,\label{Mn}\\
\frac{{d}}{{d}t}U(t,\xi)&
=\{G\ast[u_x^2+\frac{1}{2} u_{xx}^{2}]-\frac{1}{2}u^2_x\}
 \circ y\nonumber
 \\&=[\frac{1}{2}\int_{-\infty}^{+\infty}e^{-|y(t,\xi)-x|}(u_x^2+\frac{1}{2}u_{xx}^2){d}x
-\frac{1}{2}u^2_x]\circ y\nonumber
 \\&= \frac{1}{2}\int_{-\infty}^{+\infty}e^{-|y(t,\xi)-y(t,\eta)|}
 [(\frac{U_{\eta}}{y_{\eta}})^2+\frac{1}{2}(U-M)^2]y_{\eta}{d}\eta
-\frac{1}{2}(\frac{U_{\xi}}{y_{\xi}})^2
 \label{Un}\\
\frac{{d}}{{d}t}U_{\xi}(t,\xi)&=[G\ast(u_x^2+\frac{1}{2} u_{xx}^{2})-u_{x}u_{xx}]{y_{\xi}} \circ y\nonumber
 \\&=[\frac{1}{2}\int_{-\infty}^{+\infty}{ sign}\big(y(t,\xi)-x\big)e^{-|y(t,\xi)-x|}(u_x^2+\frac{1}{2}u_{xx}^2){d}x
-u_{x}u_{xx}]{y_{\xi}}\circ y\nonumber
 \\&= \frac{1}{2}\int_{-\infty}^{+\infty}sign\big(y(t,\xi)-y(t,\eta)\big)
 e^{-|y(t,\xi)-y(t,\eta)|}
 [(\frac{U_{\eta}}{y_{\eta}})^2+\frac{1}{2}(U-M)^2]y_{\eta}y_{\xi}{d}\eta
\nonumber
 \\&-U_{\xi}(U-M)
 .\label{Unxi}
\end{align}

Since $u,\ u_{x},\ m$ is uniformly bounded in $C\Big([0,T];B^{\frac{1}{p}}_{p,1}\Big)\hookrightarrow C\Big([0,T]; L^{p}\cap L^{\infty}\Big)$, we can easily deduce that $y_{\xi}$ is bounded in $L^\infty([0,T];L^\infty)$ by the Gronwall inequality. So $M(t,\xi)$, $U(t,\xi)$and $U_{\xi}(t,\xi)$ is bounded in $L^\infty([0,T];L^{\infty})$. Moreover, by \eqref{ynxi} we deduce that $\frac{1}{2}\leq y_\xi\leq C_{u_0}$ for $T>0$ small enough.
Thus, we obtain that $U(t,\xi)\in L^\infty([0,T];L^{p}\cap L^{\infty})$, $y(t,\xi)-\xi\in L^\infty([0,T];L^{p}\cap L^{\infty})$ and $\frac{1}{2}\leq y_\xi(t,\xi)\leq C_{u_0}$ for any $t\in[0,T]$.

Now we prove the uniqueness. Suppose that $m_1=(1-\partial_x^2)u_{1},m_2=(1-\partial_x^2)u_{2}$ are two solutions to \ref{E2}, then $M_{i}(t,\xi)=M_{i}(t,y_{i}(t,\xi))$, $U_{i}(t,\xi)=u_{i}(t,y_{i}(t,\xi))$,
$U_{i\xi}(t,\xi)=u_{ix}(t,y(t,\xi))y_{\xi}(t,y(t,\xi))$ satifies \eqref{yn} -\eqref{Unxi} for $i=1,2$.Hence, \eqref{E2}, \eqref{yn} -\eqref{Unxi} together with the Growall lemma yield
 \begin{align}
\frac{{d}}{{d}t}(M_1-M_2)&
=-\frac{1}{2}(M_{1}-M_{2})(M_{1}+M_{2})+U_{1}(M_{1}-M_{2})+M_{2}(U_{1}-U_{2})
\nonumber
 \\&+\frac{1}{2}\frac{1}{y_{1\xi}y_{2\xi}}(\frac{U_{1\xi}}{y_{1\xi}}
 +\frac{U_{2\xi}}{y_{2\xi}})
[y_{2\xi}(U_{1\xi}-U_{2\xi})+U_{2\xi}(y_{2\xi}-y_{1\xi})]
\nonumber
 \\&-\frac{1}{2}(U_{1}-U_{2})(U_{1}+U_{2})\nonumber
 \\&
\leq C(\|M_{1}-M_{2}\|_{L^{\infty}}
+\|U_{1}-U_{2}\|_{L^{\infty}}
+\|U_{1\xi}-U_{2\xi}\|_{L^{\infty}}
+\|y_{1\xi}-y_{2\xi}\|_{L^{\infty}}),\label{Mn1}\\
\frac{{d}}{{d}t}(U_1-U_2)&
= \frac{1}{2}\int_{-\infty}^{+\infty}
[e^{-|y_{1}(t,\xi)-y_{1}(t,\eta)|}-e^{-|y_{2}(t,\xi)-y_{2}(t,\eta)|}]
 (\frac{U_{1\eta}}{y_{1\eta}})^2{d}\eta\nonumber
 \\&+\frac{1}{2}\int_{-\infty}^{+\infty}e^{-|y_{2}(t,\xi)-y_{2}(t,\eta)|}
 [(\frac{U_{1\eta}}{y_{1\eta}})^2-(\frac{U_{2\eta}}{y_{2\eta}})^2]
 {d}\eta
 \nonumber
 \\&+\frac{1}{2}\int_{-\infty}^{+\infty}
[e^{-|y_{1}(t,\xi)-y_{1}(t,\eta)|}-e^{-|y_{2}(t,\xi)-y_{2}(t,\eta)|}]
 (U_{1}-M_{1})^2y_{1\eta}{d}\eta\nonumber
 \\&+\frac{1}{2}\int_{-\infty}^{+\infty}e^{-|y_{2}(t,\xi)-y_{2}(t,\eta)|}
 [(U_{1}-M_{1})^2y_{1\eta}-(U_{2}-M_{2})^2y_{2\eta}]
 {d}\eta
 \nonumber
 \\&+\frac{1}{2}\frac{1}{y_{1\xi}y_{2\xi}}(\frac{U_{1\xi}}{y_{1\xi}}
 +\frac{U_{2\xi}}{y_{2\xi}})
[y_{2\xi}(U_{1\xi}-U_{2\xi})+U_{2\xi}(y_{2\xi}-y_{1\xi})]\nonumber
 \\&\leq C(\|M_{1}-M_{2}\|_{L^{\infty}}
+\|U_{1}-U_{2}\|_{L^{\infty}}
+\|U_{1\xi}-U_{2\xi}\|_{L^{\infty}}
+\|y_{1\xi}-y_{2\xi}\|_{L^{\infty}}).
 \label{Un1}
\end{align}

Since $y_i(i=1,2)$ is monotonically increasing, then ${\rm sign}\big(y_i(\xi)-y_i(\eta)\big)={\rm sign}\big(\xi-\eta\big)$. Thus, we have
\begin{align}
\frac{{d}}{{d}t}(U_{1\xi}-U_{2\xi})&= \frac{1}{2}\int_{-\infty}^{+\infty}[sign\big(\xi-\eta)\big)
 e^{-|y_{1}(t,\xi)-y_{1}(t,\eta)|}-sign\big(\xi-\eta)\big)
  e^{-|y_{2}(t,\xi)-y_{2}(t,\eta)|}]\nonumber
 \\&
[(\frac{U_{1\eta}}{y_{1\eta}})^2+\frac{1}{2}(U_{1}-M_{1})^2]
 y_{1\eta}y_{1\xi}{d}\eta
\nonumber
 \\&+\frac{1}{2}\int_{-\infty}^{+\infty}sign\big(\xi-\eta)\big)
 e^{-|y_{2}(t,\xi)-y_{2}(t,\eta)|}
\nonumber
 \\& [(\frac{U_{1\eta}}{y_{1\eta}})^2-(\frac{U_{2\eta}}{y_{2\eta}})^2]
 +\frac{1}{2}[(U_{1}-M_{1})^2-(U_{2}-M_{2})^2]{d}\eta
 \nonumber
 \\&-\{U_{1\xi}[(U_{1}-M_{1})-(U_{2}-M_{2})]+(U_{2}-M_{2})(U_{1\xi}-U_{2\xi})\}
 \nonumber
 \\&\triangleq I_1+I_2+I_3.\label{Unxi1}
\end{align}
If $\xi>\eta$ $(\text{or}\ \xi<\eta)$, then $y_i(\xi)>y_i(\eta)$ $\big(\text{or}\ y_i(\xi)<y_i(\eta)\big)$. Hence, we gain
\begin{align}
I_1=&-\int_{-\infty}^{\xi}\left(e^{-(y_1(\xi)-y_1(\eta))}
-e^{-(y_2(\xi)-y_2(\eta))}\right)
[(\frac{U_{1\eta}}{y_{1\eta}})^2+\frac{1}{2}(U_{1}-M_{1})^2]
 y_{1\eta}y_{1\xi}{d}\eta
\notag\\
&
+\int_{\xi}^{+\infty}\left(e^{y_1(\xi)-y_1(\eta)}
-e^{y_2(\xi)-y_2(\eta)}\right)[(\frac{U_{1\eta}}{y_{1\eta}})^2
+\frac{1}{2}(U_{1}-M_{1})^2]
 y_{1\eta}y_{1\xi}{d}\eta\notag\\
=&\int_{-\infty}^{\xi}e^{-(\xi-\eta)}
\left(e^{-\int_0^t(\frac{U_{1\xi}}{y_{1\xi}}
-\frac{U_{1\eta}}{y_{1\eta}}){d}\tau}
-e^{-\int_0^t(\frac{U_{2\xi}}{y_{2\xi}}
-\frac{U_{2\eta}}{y_{2\eta}}){d}\tau}\right)
[(\frac{U_{1\eta}}{y_{1\eta}})^2
+\frac{1}{2}(U_{1}-M_{1})^2]y_{1\eta}y_{1\xi}{d}\eta\notag\\
&-\int_{\xi}^{+\infty}e^{\xi-\eta}\left(e^{\int_0^t (\frac{U_{1\xi}}{y_{1\xi}}
-\frac{U_{1\eta}}{y_{1\eta}}){d}\tau}
-e^{\int_0^t(\frac{U_{1\xi}}{y_{1\xi}}
-\frac{U_{1\eta}}{y_{1\eta}}){d}\tau}\right)
[(\frac{U_{1\eta}}{y_{1\eta}})^2
+\frac{1}{2}(U_{1}-M_{1})^2]y_{1\eta}y_{1\xi}{d}\eta\notag\\
\leq&C(\|U_{1\eta}-U_{2\eta}\|_{L^\infty}+\|y_{1\eta}-y_{2\eta}\|_{L^\infty})
\Big[\int_{-\infty}^{\xi}e^{-(\xi-\eta)}
[(\frac{U_{1\eta}}{y_{1\eta}})^2
\notag\\&+\frac{1}{2}(U_{1}-M_{1})^2]y_{1\eta}y_{1\xi}{d}\eta+\int_{\xi}^{+\infty}
e^{\xi-\eta}[(\frac{U_{1\eta}}{y_{1\eta}})^2
+\frac{1}{2}(U_{1}-M_{1})^2]y_{1\eta}y_{1\xi}{d}\eta\Big]\notag\\
\leq&C(\|U_{1\eta}-U_{2\eta}\|_{L^\infty}+\|y_{1\eta}-y_{2\eta}\|_{L^\infty})\Big[1_{\geq0}(x)e^{-|x|}
\ast[(\frac{U_{1\eta}}{y_{1\eta}})^2
\notag\\&+\frac{1}{2}(U_{1}-M_{1})^2]y_{1\eta}y_{1\xi}+1_{\leq0}(x)e^{-|x|}
\ast[(\frac{U_{1\eta}}{y_{1\eta}})^2
+\frac{1}{2}(U_{1}-M_{1})^2]y_{1\eta}y_{1\xi}\Big].\label{i1}
\notag\\
\leq&C(\|U_{1\eta}-U_{2\eta}\|_{L^\infty}+\|y_{1\eta}-y_{2\eta}\|_{L^\infty})\end{align}
In the same way, we have
\begin{align}
&I_2\leq C(\|M_{1}-M_{2}\|_{L^{\infty}}
+\|U_{1}-U_{2}\|_{L^{\infty}}
+\|U_{1\xi}-U_{2\xi}\|_{L^{\infty}}
+\|y_{1\xi}-y_{2\xi}\|_{L^{\infty}})\label{I2}\\
&I_3\leq C(\|M_{1}-M_{2}\|_{L^{\infty}}
+\|U_{1}-U_{2}\|_{L^{\infty}}
+\|U_{1\xi}-U_{2\xi}\|_{L^{\infty}}).\label{I3}
\end{align}
Hence,  \eqref{Mn1}-\eqref{I3} together with the Growall lemma yield
\begin{align*}
&\quad\|M_{1}-M_{2}\|_{L^{\infty}}
+\|U_{1}-U_{2}\|_{L^{\infty}}
+\|U_{1\xi}-U_{2\xi}\|_{L^{\infty}}
+\|y_{1\xi}-y_{2\xi}\|_{L^{\infty}}\notag\\
&\leq C(\|M_{1}(0)-M_{2}(0)\|_{L^{\infty}}
+\|U_{1}(0)-U_{2}(0)\|_{L^{\infty}}
\notag\\
& \quad +\|U_{1\xi}(0)-U_{2\xi}(0)\|_{L^{\infty}}
+\|y_{1\xi}(0)-y_{2\xi}(0)\|_{L^{\infty}}
)\notag\\
&+C\int_{0}^{T}\quad\|M_{1}-M_{2}\|_{L^{\infty}}
+\|U_{1}-U_{2}\|_{L^{\infty}}
\notag\\&+\|U_{1\xi}-U_{2\xi}\|_{L^{\infty}}
+\|y_{1\xi}-y_{2\xi}\|_{L^{\infty}\cap L^{p}}{d}t\notag\\
&\leq C(\|M_{1}(0)-M_{2}(0)\|_{B^{\frac{1}{p}}_{p,1}}
+\|U_{1}(0)-U_{2}(0)\|_{B^{\frac{1}{p}}_{p,1}}
+\|U_{1\xi}(0)-U_{2\xi}(0)\|_{B^{\frac{1}{p}}_{p,1}})\notag\\
&\leq C\|m_{1}(0)-m_{2}(0)\|_{B^{\frac{1}{p}}_{p,1}},
\end{align*}
where $y_1(0)=y_2(0)=\xi$, $y_{1\xi}(0)=y_{2\xi}(0)=1,$ $M_i(0)=m_i(0),U_i(0)=G\ast m_i(0),U_{i\xi}(0)=G_{\xi}\ast m_i(0)~i=1,2$.

It follows that
\begin{align}
	\|u_{1}-u_{2}\|_{L^\infty}\leq& C\|u_{1}\circ y_{1}-u_{2}\circ y_{1}\|_{L^\infty}\nonumber\\
	\leq& C\|u_{1}\circ y_{1}-u_{2}\circ y_{2}+u_{2}\circ y_{2}-u_{2}\circ y_{1}\|_{L^\infty}\nonumber\\
	\leq& C\|U_{1}-U_{2}\|_{L^\infty}+C\|u_{2x}\|_{L^\infty}\|y_{1}-y_{2}\|_{L^\infty}
\nonumber\\
	\leq& C\|m_{1}(0)-m_{2}(0)\|_{B^{\frac{1}{p}}_{p,1}}.
\label{Un2}\\
	\|u_{1x}-u_{2x}\|_{L^\infty}\leq& C\|u_{1x}\circ y_{1}-u_{2x}\circ y_{1}\|_{L^\infty}\nonumber\\
	\leq& C\|u_{1x}\circ y_{1}-u_{2x}\circ y_{2}+u_{2x}\circ y_{2}-u_{2x}\circ y_{1}\|_{L^\infty}\nonumber\\
	\leq& C(\|\frac{U_{1\xi}}{y_{1\xi}}-\frac{U_{2\xi}}{y_{2\xi}}\|_{L^\infty}
+\|u_{2xx}\|_{L^\infty}\|y_{1}-y_{2}\|_{L^\infty})\nonumber\\
\leq& C(\|U_{1\xi}-U_{2\xi}\|_{L^\infty}+\|y_{1\xi}-y_{2\xi}\|_{L^\infty}
+\|u_{2xx}\|_{L^\infty}\|y_{1}-y_{2}\|_{L^\infty})\nonumber\\
	\leq& C\|m_{1}(0)-m_{2}(0)\|_{B^{\frac{1}{p}}_{p,1}}\label{Un3}.
\end{align}
So if $m_1(0)=m_2(0)$, we can immediately obtain that $u_{1}=u_{2},u_{1x}=u_{2x}$ $a.e\ in\  \mathbb{R}$.

Set $W=m_1-m_2,$ $u_{1}=u_{2}=v,\ u_{1x}=u_{2x}=v_{x}.$
 Hence, we obtain that
\begin{align}
\left\{
\begin{array}{ll}
&\partial_{t}W-v_{x}\partial_{x}W
 =\frac{1}{2}W (2v-m_{1}-m_{2}),\\[1ex]
&W|_{t=0}=m_1(0)=m_2(0),\\[1ex]
\end{array}
\right.
\end{align}

Applying Lemma  \ref{priori estimate} and Gronwall inequality, we deduce that
\begin{align}
\|m_{1}-m_{2}\|_{L^\infty}
\leq C\|m_{1}(0)-m_{2}(0)\|_{B^{\frac{1}{p}}_{p,1}}\label{U1}.
\end{align}
By the embedding $L^\infty\hookrightarrow B^0_{\infty,\infty}$, we get
\begin{align*}
	\|m_{1}-m_{2}\|_{B^0_{\infty,\infty}}\leq C\|m_{1}-m_{2}\|_{L^\infty}\leq C\|m_{1}(0)-m_{2}(0)\|_{B^{\frac{1}{p}}_{p,1}}.	
\end{align*}

\textbf{Step 3. The continuous dependence.}
Then we prove the solution of (\ref{E2}) guaranteed by Theorem \ref{Thm1} depends continuously on the initial data.

Assume that $m^n_{0}$ tends to $m^\infty_{0}$ in $B^{\frac{1}{p}}_{p,1}$,  $u^n_{0}$ tends to $u^\infty_{0}$ in $B^{2+\frac{1}{p}}_{p,1}$ and $m^n,\ m^\infty$ are the solutions of \ref{E2} with the initial data $m^n_{0}, m^\infty_{0}$ respectively. Similar to \cite{ye1}, we can find the solution of \ref{E2} with a common lifespan $T$.
By \textbf{Step 1--Step 2}, we have $m^n,\ m^\infty$ are unformly bounded in $L^{\infty}([0,T];B^{\frac{1}{p}}_{p,1})$, $\ u^n,\ u^\infty$ are unformly bounded in $L^{\infty}([0,T];B^{2+\frac{1}{p}}_{p,1})$  and
\begin{align*}
&\|\big(u^n-u^\infty\big)(t)\|_{B^0_{\infty,\infty}}\leq  C\|m_{0}^n-m_{0}^\infty\|_{B^{\frac{1}{p}}_{p,1}},\ \forall t\in[0,T].
\\&\|\big(u_{x}^n-u_{x}^\infty\big)(t)\|_{B^0_{\infty,\infty}}\leq  C\|m_{0}^n-m_{0}^\infty\|_{B^{\frac{1}{p}}_{p,1}},\ \forall t\in[0,T].
\end{align*}
Taking advantage of the interpolation inequality, we see that
\begin{align*}
&u^n\rightarrow u^\infty, \text{in} C\big([0,T];B^{\frac{1}{p}}_{p,1}\big),
\\&u_{x}^n\rightarrow u_{x}^\infty, \text{in} C\big([0,T];B^{\frac{1}{p}}_{p,1}\big).
\end{align*}


we next only need to prove $m^n\rightarrow m^\infty$ in $C\big([0,T];B^{\frac{1}{p}}_{p,1}\big)$.
Split $m^n$ into $w^n+z^n$ with $(w^n,z^n)$ satisfying
\begin{equation*}
	\left\{\begin{array}{l} \partial_{t}w^n+\partial_{x}u^n\partial_{x}w^n=F(m^{\infty},u^{\infty},u^{\infty}),\\
		w^n(0,x)=m^\infty_0
	\end{array}\right.
\end{equation*}
and
\begin{equation*}
\left\{\begin{array}{l}
\partial_{t}z^n+\partial_{x}u^{n}\partial_{x}z^n=F(m^{n},u^{n})
-F(m^{\infty},u^{\infty})\\
z^n(0,x)=v^n_0-v^\infty_0=m^n_0-m^\infty_0.
\end{array}\right.
\end{equation*}
\eqref{12} and Lemma \ref{continuous} thus ensure that
\begin{align}
	w^n\rightarrow w^\infty\qquad \text{in}\quad  C\big([0,T];B^{\frac{1}{p}}_{p,1}\big).\label{winfty}
\end{align}
Thanks to \eqref{12}, we have
\begin{align*}
\|F(m^{n},u^{n})-F(m^{\infty},u^{\infty})\|_{B^{\frac{1}{p}}_{p,1}}
\leq C(\|m^n-m^\infty\|_{B^{\frac{1}{p}}_{p,1}}
+\|u^n-u^\infty\|_{B^{\frac{1}{p}}_{p,1}}
+\|u^n-u^\infty\|_{B^{\frac{1}{p}}_{p,1}}).
\end{align*}

It follows that for all $n\in\mathbb{N}$,
\begin{align}
	\|z^n(t)\|_{B^{\frac{1}{p}}_{p,1}}	\leq&C\left(\|m^n_0-m^\infty_0\|_{B^{\frac{1}{p}}_{p,1}}
+\int_0^t\|m^n-m^\infty\|_{B^{\frac{1}{p}}_{p,1}}+\|u^n-u^\infty\|_{B^{\frac{1}{p}}_{p,1}}
+\|u^n_x-u^\infty_x\|_{B^{\frac{1}{p}}_{p,1}}{d}\tau\right)\notag\\	\leq&C\left(\|m^n_0-m^\infty_0\|_{B^{\frac{1}{p}}_{p,1}}
+\int_0^t\|u^n_x-u^\infty_x\|_{B^{\frac{1}{p}}_{p,1}}+\|u^n-u^\infty\|_{B^{\frac{1}{p}}_{p,1}}
+\|w^n-w^\infty\|_{B^{\frac{1}{p}}_{p,1}}
+\|z^n\|_{B^{\frac{1}{p}}_{p,1}}{d}\tau\right).\label{zn}
\end{align}
Using the facts that
\begin{itemize}
	\item [-]$m_0^n$ tends to $m_0^{\infty}$ in $B^{\frac{1}{p}}_{p,1}$;
	\item [-]$u^n$ tends to $u^{\infty}$ in $C\big([0,T];B^{\frac{1}{p}}_{p,1}\big)$;
\item [-]$u_x^n$ tends to $u_x^{\infty}$ in $C\big([0,T];B^{\frac{1}{p}}_{p,1}\big)$;
	\item [-]$w^n$ tends to $w^{\infty}$ in $C\big([0,T];B^{\frac{1}{p}}_{p,1}\big)$,
\end{itemize}
and then applying the Gronwall lemma, we conclude that $z^n$ tends to $0$ in $C\big([0,T];B^{\frac{1}{p}}_{p,1}\big)$. By Lemma \ref{existence}--\ref{priori estimate}, we have
$z^\infty=0$ in $C\big([0,T];B^{\frac{1}{p}}_{p,1}\big)$.

Therefore,
\begin{align*}
	\|m^n-m^\infty\|_{L^\infty\big([0,T];B^{\frac{1}{p}}_{p,1}\big)}
\leq&\|w^n-w^\infty\|_{L^\infty\big([0,T];B^{\frac{1}{p}}_{p,1}\big)}
+\|z^n-z^\infty\|_{L^\infty\big([0,T];B^{\frac{1}{p}}_{p,1}\big)}\\
	\leq&\|w^n-w^\infty\|_{L^\infty\big([0,T];B^{\frac{1}{p}}_{p,1}\big)}
+\|z^n\|_{L^\infty\big([0,T];B^{\frac{1}{p}}_{p,1}\big)}\quad\rightarrow 0\quad\text{as}\ n\rightarrow\infty,
\end{align*}
that is
\begin{align*}
	m^n\rightarrow m^\infty\qquad\text{in}\quad C\big([0,T];B^{\frac{1}{p}}_{p,1}\big).
\end{align*}
Hence, we prove the continuous dependence of \eqref{E2} in critial Besov spaces $C\big([0,T];B^{\frac{1}{p}}_{p,1}\big)$ with $p\in[1,+\infty)$.

Consequently, combining with \textbf{Step 1--Step 3}, we finish the proof of Theorem \ref{Thm1}.

\end{proof}

\bigskip

{\bf Acknowledgements}.  This work was
partially supported by NNSFC (No. 11671407 and No. 11801076),FDCT (No. 0091/2013/A3), Guangdong Special Support Program (No. 8-2015)
and the key project of NSF of Guangdong Province (No. 2016A030311004).

\phantomsection
\addcontentsline{toc}{section}{\refname}

\end{document}